\documentclass[reqno]{amsart}
\usepackage{latexsym,amssymb,amsthm,amsmath}

\usepackage{amssymb}
\usepackage{amsmath}

\theoremstyle{plain}
\newtheorem{theorem}{Theorem}[section]
\newtheorem*{Theorem B}{Theorem B}
\newtheorem*{Theorem A}{Theorem A}
\newtheorem{lemma}{Lemma}[section]

\newtheorem{corollary}{Corollary}[section]

\numberwithin{equation}{section}

\theoremstyle{remark}

\begin{document}

\title[Hamilton-Souplet-Zhang type estimations along geometric flow]
{Hamilton and Souplet-Zhang type estimations on semilinear parabolic system along geometric flow}

\author[S. K. Hui, S. Azami and S. Bhattacharyya]{Shyamal Kumar Hui, Shahroud Azami and Sujit Bhattacharyya}


\begin{abstract}
In this article we derive both Hamilton type and Souplet-Zhang type gradient estimations for a system of semilinear equations along a geometric flow on a weighted Riemannian manifold.
\end{abstract}

\maketitle

\section{Introduction}
The study of gradient estimation was started after the work of Li and Yau \cite{Li-Yau}, where they derived a bound for the quantity $\frac{|\nabla u|}{u}$ and $u$ satisfies
\begin{eqnarray}
	\nonumber \left(\Delta-q(x)-\partial_t\right)u(x,t)=0,
\end{eqnarray}which is known as the Li-Yau type estimation. This field becomes much popular after introducing Ricci flow by Hamilton \cite{Hamilton-1,Hamilton}. Later Souplet and Zhang \cite{SOUPLET-ZHANG} developed this area. Gradient estimations were studied on many different system of equations. In \cite{SHEN-DING}, Shen and Ding considered the following system
\begin{eqnarray}
	\begin{cases}
	 \nonumber u_t =\Delta u^m+k_1(t)f_1(v),\\
	 \nonumber v_t =\Delta v^n+k_2(t)f_2(u),
	\end{cases}
\end{eqnarray}
with nonlinear boundary conditions, where $\Delta$ is the Laplace Beltrami operator and proved that the solution of the above system blows up in finite time using differential inequality and Sobolev inequality. A parabolic system of the form
\begin{eqnarray}
	\begin{cases}
	\nonumber u_t = \Delta u-\nabla \cdot (u\nabla v),\\
	\nonumber v_t = \Delta -v+u,\\
	\end{cases}
\end{eqnarray}
is called Keller-Segal system, which was studied by Winkler \cite{WINKLER}. Global existence and finite time blow up of the solution for the following semilinear parabolic system
\begin{eqnarray}
	\begin{cases}
	\nonumber u_t = \Delta u+e^{\alpha t}v^p,\\
	\nonumber v_t = \Delta v+e^{\beta t}u^q,
	\end{cases}
\end{eqnarray}
on hyperbolic space was showed by Wu and Yang \cite{Wu-Yang}. System of nonlinear parabolic equations have numerous applications in quantum physics, fluid dynamics, laser-plasma interaction, study of Navier-Stokes equation and many other fields. For example, the system of Zakharov equations plays an important role in laser-plasma interaction. Various interesting results for this system have been studied by Zheng, Shang and Di \cite{Zheng-Shang}. Uniform stability for Memory-Type Elasticity System have been studied by Li and Bao \cite{Li-Bao}. In \cite{Feng-Liu}, Feng and Liu studied smooth solutions of compressible Navier-Stokes-Poisson system in $\mathbb{R}^3$, which represents the dynamics of electrons in semiconductors.

 Motivated by the works of Wu \cite{WU}, in this article, we acquire the Hamilton type and Souplet-Zhang type gradient estimations for the system of weighted semilinear equations
\begin{eqnarray}\label{heat eqn 1.1}
		\begin{cases}
		\begin{aligned}
		(\Delta_\phi -\partial_t)f &=& -e^{\lambda_1 t}h^p,\\
		(\Delta_\phi -\partial_t)h &=& -e^{\lambda_2 t}f^q,
		\end{aligned}
		\end{cases}
\end{eqnarray}
on a weighted Riemannian manifold $(M^n,g,e^{-\phi}d\mu)$ along an abstract geometric flow
\begin{eqnarray}\label{flow 1.2}
	\frac{\partial}{\partial t}g_{ij}=2S_{ij},
\end{eqnarray} where $g(t)$ is an one parameter family of the Riemannian metric evolving along (\ref{flow 1.2}), $e^{-\phi}d\mu$ is the weighted volume form, $\phi$ is a twice differentiable function on $M$, $p,q,\lambda_1,\lambda_2$ are positive constants and $\Delta_\phi:=\Delta-\nabla\phi\nabla$ denotes the weighted Laplacian operator, $S_{ij}(t):=\mathcal{S}(e_i,e_j)$ is a smooth symmetric 2-tensor on $(M,g(t))$.
\section{Preliminaries}
This section is devoted to the basic results and evolution equations that will be used later.
\begin{lemma}[Weighted Bochner formula]\cite{SAZAMI}\label{lemma Bochner} For any smooth function $u$, we have
	\begin{eqnarray}
		\nonumber \frac12 \Delta_\phi |\nabla u|^2 &=& |\text{Hess }u|^2+\langle\nabla \Delta_\phi u,\nabla u\rangle+Ric_\phi(\nabla u,\nabla u),
	\end{eqnarray}
	where $Ric_\phi:=Ric+\text{Hess }\phi$, is called the Bakry-\'Emery Ricci tensor \cite{Bakry-Emery}, \text{Hess }is the Hessian operator and for any integer $m>n$, the $(m-n)$-Bakry-\'Emery Ricci tensor is defined by
	\begin{eqnarray}
	\nonumber Ric^{m-n}_\phi:=Ric+\text{Hess }\phi-\frac{\nabla\phi\otimes\nabla\phi}{m-n}.
	\end{eqnarray}
\end{lemma}
\begin{lemma}\cite{SAZAMI}\label{lemma evol eqn}
	If $u$ is any smooth function on $M$ then under the geometric flow (\ref{flow 1.2}), the expression $|\nabla u|^2$ evolves by
	\begin{eqnarray}
		\frac{\partial}{\partial t}|\nabla u|^2 &=& -2\mathcal{S}(\nabla u,\nabla u)+2\langle\nabla u,\nabla u_t\rangle.
	\end{eqnarray}
\end{lemma}
Let $x_0\in M$ and $R>0$, $T>0$ be any two real numbers. Define a compact subset of $M\times (-\infty,\infty)$ by $$D_T(2R)=\left\{(x,t):d_t(x,x_0)\le 2R,\ 0\le t\le T\right\},$$ where $d_t(x,x_0)$ is the geodesic distance between $x$ and $x_0$ under the metric $g(t)$.
\section{Hamilton type gradient estimate}
In this section we derive both local and global Hamilton type estimation for positive solution of (\ref{heat eqn 1.1}) along (\ref{flow 1.2}), for this we consider $(f,h)=(u^3,v^3)$ as a positive solution to the system (\ref{heat eqn 1.1}). Substituting $f=u^3$, $h=v^3$ in (\ref{heat eqn 1.1}) we get
\begin{eqnarray}\label{heat eqn reduced}
	\begin{cases}
	\begin{aligned}
		3u^2(\Delta_\phi -\partial_t)u &=& -6u|\nabla u|^2-e^{\lambda_1 t}v^{3p},\\
		3v^2(\Delta_\phi -\partial_t)v &=& -6v|\nabla v|^2-e^{\lambda_2 t}u^{3q},
	\end{aligned}
	\end{cases}
\end{eqnarray}
\begin{lemma}
	Let $(u,v)$ be a positive solution to the system of equations (\ref{heat eqn reduced}). If there exist positive constants $k_1,k_2$ such that $Ric_\phi\ge -(n-1)k_1 g$ and $\mathcal{S}\ge -k_2g$ then the function $\mathcal{H}_1:=u|\nabla u|^2$ satisfies
	\begin{eqnarray}\label{lemma eqn u}
	\nonumber (\Delta_\phi-\partial_t)\mathcal{H}_1 &\ge& 4u^{-3}\mathcal{H}_1^2-4u^{-1}\langle\nabla u,\nabla \mathcal{H}_1\rangle -2pe^{\lambda_1t}v^{3p-\frac32}u^{-\frac32}\sqrt{\mathcal{H}_1\mathcal{H}_2}\\
	&&+e^{\lambda_1 t}v^{3p}u^{-2}\mathcal{H}_1-2\left((n-1)k_1+k_2\right)\mathcal{H}_1
	\end{eqnarray}
and the function $\mathcal{H}_2:=v|\nabla v|^2$ satisfies
\begin{eqnarray}\label{lemma eqn v}
	\nonumber (\Delta_\phi-\partial_t)\mathcal{H}_2 &\ge& 4v^{-3}\mathcal{H}_2^2-4v^{-1}\langle\nabla v,\nabla \mathcal{H}_2\rangle -2qe^{\lambda_2t}u^{3q-\frac32}v^{-\frac32}\sqrt{\mathcal{H}_2\mathcal{H}_1}\\
	&&+e^{\lambda_2 t}u^{3q}v^{-2}\mathcal{H}_2-2\left((n-1)k_1+k_2\right)\mathcal{H}_2.
\end{eqnarray}
\end{lemma}
\begin{proof}
	Differentiating $\mathcal{H}_1=u|\nabla u|^2$ with respect to $t$ gives
	\begin{eqnarray}\label{3.4}
		\partial_t \mathcal{H}_1 &=& |\nabla u|^2u_t+u\partial_t|\nabla u|^2.
	\end{eqnarray}
In view of Lemma \ref{lemma evol eqn}, (\ref{3.4}) gives
\begin{eqnarray}\label{3.5}
	\partial_t \mathcal{H}_1 &=& |\nabla u|^2u_t-2u\mathcal{S}(\nabla u,\nabla u)+2u\langle\nabla u,\nabla u_t\rangle.
\end{eqnarray}
Using weighted Bochner formula (Lemma \ref{lemma Bochner}) we have
\begin{eqnarray}\label{3.6}
	\nonumber \Delta_\phi\mathcal{H}_1&=&|\nabla u|^2\Delta_\phi u+2u|\text{Hess }u|^2+2u\langle\nabla \Delta_\phi u,\nabla u\rangle+2uRic_\phi(\nabla u,\nabla u)\\
	&&+4\text{Hess }u(\nabla u,\nabla u).
\end{eqnarray}
Subtracting (\ref{3.5}) from (\ref{3.6}) we have
\begin{eqnarray}\label{3.7}
	\nonumber (\Delta_\phi-\partial_t)\mathcal{H}_1 &=& |\nabla u|^2(\Delta_\phi-\partial_t) u+2u|\text{Hess }u|^2+2u\langle\nabla (\Delta_\phi-\partial_t) u,\nabla u\rangle\\
	&&+2u(Ric_\phi+\mathcal{S})(\nabla u,\nabla u)+4\text{Hess }u(\nabla u,\nabla u).
\end{eqnarray}
Applying the first equation of the system (\ref{heat eqn reduced}) in (\ref{3.7}) we have
\begin{eqnarray}\label{3.8}
	\nonumber (\Delta_\phi-\partial_t)\mathcal{H}_1 &=& -\frac13 |\nabla u|^2e^{\lambda_1 t}v^{3p}u^{-2}+2u|\text{Hess }u|^2+4\text{Hess }u(\nabla u,\nabla u)\\
	\nonumber &&+\frac{2}{u}|\nabla u|^4-8\text{Hess }u(\nabla u,\nabla u)-2pe^{\lambda_1 t}v^{3p-1}u^{-1}\langle\nabla u,\nabla v\rangle\\
	&&+\frac43 e^{\lambda_1 t}v^{3p}u^{-2}|\nabla u|^2+2u(Ric_\phi+\mathcal{S})(\nabla u,\nabla u).
\end{eqnarray}
Note that 
\begin{eqnarray}
	\nonumber 2u|\text{Hess }u|^2+4\text{Hess }u(\nabla u,\nabla u)+\frac{2}{u}|\nabla u|^4 &=& 2u\left|\text{Hess }u+\frac{\nabla u\otimes\nabla u}{u}\right|^2\\
	&\ge& 0.
\end{eqnarray}
Hence (\ref{3.8}) reduces to 
\begin{eqnarray}\label{3.10}
	\nonumber (\Delta_\phi-\partial_t)\mathcal{H}_1 &\ge& e^{\lambda_1 t}v^{3p}u^{-2}\mathcal{H}_1-8\text{Hess }u(\nabla u,\nabla u)-2pe^{\lambda_1 t}v^{3p-1}u^{-1}|\nabla u||\nabla v|\\
	\nonumber &&+2u(Ric_\phi+\mathcal{S})(\nabla u,\nabla u)\\
	\nonumber &=& e^{\lambda_1 t}v^{3p}u^{-2}\mathcal{H}_1-8\text{Hess }u(\nabla u,\nabla u)-2pe^{\lambda_1 t}v^{3p-\frac32}u^{-\frac32}\sqrt{\mathcal{H}_1\mathcal{H}_2}\\
	 &&+2u(Ric_\phi+\mathcal{S})(\nabla u,\nabla u).
\end{eqnarray}
We see that
\begin{eqnarray}\label{3.11}
	8\text{Hess }u(\nabla u,\nabla u) &=& 4u^{-1}\langle\nabla u,\nabla \mathcal{H}_1\rangle-4u^{-3}\mathcal{H}_1^2.
\end{eqnarray}
Applying (\ref{3.11}) and the bounds of $Ric_\phi$, $\mathcal{S}$ in (\ref{3.10}) we get (\ref{lemma eqn u}).\\
\\
Due to symmetry we can easily obtain the result (\ref{lemma eqn v}) from (\ref{lemma eqn u}).\\
This completes the proof.
\end{proof}
\noindent Let $\tilde{\kappa}_1$, $\kappa_1$, $\tilde{\kappa}_2$, $\kappa_2$, $k_1$, $k_2$ be positive constants. 
Let $T_1\in(0,T]$ and $(x_1,t_1)\in D_{T_1}(2R)$ be any point.
For any $x_1\in M$ and $R>0$ we may find a cut-off function $\psi:[0,\infty)\to [0,1]$ defined by \cite{Li-Yau}
\begin{eqnarray}
	\label{eq psi defn}	\psi(s) &=& 
	\begin{cases}
			1,\ s\in[0,1],\\
			0,\ s\in[2,\infty),
	\end{cases}
\end{eqnarray}
belonging to $C^2(M)$ satisfying $-c_0\le \psi'(s)\le 0$, $\psi''(s)\ge -c_1$ and $\frac{|\psi''(s)|^2}{\psi(s)}\le c_1$, where $c_0$, $c_1$ are positive constants. Let $R>1$ be a constant and define 
\begin{eqnarray}
	\label{eq eta}\eta(x,t) &=& \psi\left(\frac{d_t(x,x_1)}{R}\right).
\end{eqnarray}
To assume everywhere smoothness of $\psi$ so that we can use maximum principle, we apply Calabi's argument \cite{CALABI}.
By generalized Laplacian comparison theorem \cite{Li,SCOHEN,WU-2} we get
\begin{enumerate}
	\item $\Delta_\phi d_t(x,x_1)\le(n-1)\sqrt{k_1}\coth(\sqrt{k_1}d_t(x,x_1))$,
	\item $\Delta_\phi \eta \ge -\frac{c_0}{R}(n-1)(\sqrt{k_1}+\frac2R)-\frac{c_1}{R^2}$,
	\item $\frac{|\nabla \eta|^2}{\eta}\le\frac{c_1}{R^2}$.
\end{enumerate}
Let 
\begin{eqnarray}
	\label{eq G1G2}\mathcal{G}_1=\eta \mathcal{H}_1,\ \mathcal{G}_2=\eta \mathcal{H}_2.
\end{eqnarray}
We denote
\begin{enumerate}
	\item[] $\Omega_1(x,t)=\Omega u^3-e^{\lambda_1 t}v^{3p}u\eta+2((n-1)k_1+k_2)\eta u^3$,
	\item[] $\Omega_2(x,t)=\frac{4\sqrt{c_1}}{R}u^\frac32$,
	\item[] $\Omega_3(x,t)=2pe^{\lambda_1 t}v^{3p-\frac32}u^\frac32 \eta \sqrt{\mathcal{G}_2}$,
	\item[] $\bar{\Omega}_1(x,t)=\Omega v^3-e^{\lambda_2 t}u^{3q}v\eta+2((n-1)k_1+k_2)\eta v^3$,
	\item[] $\bar{\Omega}_2(x,t)=\frac{4\sqrt{c_1}}{R}v^\frac32$,
	\item[] $\bar{\Omega}_3(x,t)=2qe^{\lambda_2 t}u^{3q-\frac32}v^\frac32 \eta \sqrt{\mathcal{G}_1}$,
\end{enumerate}	
and 
\begin{enumerate}
		\item[] $\Omega=\frac{c_0}{R}(m-1)(\sqrt{k_1}+\frac2R)+\frac{3c_1}{R^2}+c_2k_2$, 
		\item[] $\Omega_1^*=\Omega \kappa_1^3+2((n-1)k_1+k_2) \kappa_1^3$,
		\item[] $\Omega_2^*=\frac{4\sqrt{c_1}}{R}\kappa_1^\frac32$,
		\item[] $\bar{\Omega}_1^*=\Omega \kappa_2^3+2((n-1)k_1+k_2)\kappa_2^3$,
		\item[] $\bar{\Omega}_2^*=\frac{4\sqrt{c_1}}{R}\kappa_2^\frac32$,
		\item[] $\hat{\Omega}_1^*=c_2k_2\kappa_1^3+2\kappa_1^3((n-1)k_1+k_2)$,
		\item[] $\hat{\bar{\Omega}}_1^*=c_2k_2\kappa_2^3+2\kappa_2^3((n-1)k_1+k_2)$,
\end{enumerate}
where $\eta$, $\mathcal{G}_1$, $\mathcal{G}_2$ are defined in (\ref{eq eta}) and (\ref{eq G1G2}) respectively.
\begin{theorem}\label{Theorem 3.1}
	If $(f,h)$ is a positive solution to the system (\ref{heat eqn 1.1}) along the flow (\ref{flow 1.2}) satisfying $\tilde{\kappa}_1^3\le f\le \kappa_1^3$ and $\tilde{\kappa}_2^3\le h\le \kappa_2^3$ in $D_T(2R)$ and $Ric_\phi\ge -(n-1)k_1g$, $\mathcal{S}\ge -k_2 g$ on $D_T(2R)$ with $t\in[0,T]$ then
	\begin{eqnarray}\label{Theorem 3.1 result}
	\begin{cases}
	\frac{|\nabla f|}{\sqrt{f}} \le 3\sqrt{\frac{\kappa_1}{\tilde{\kappa}_1}}\left(\sqrt[4]{\frac{l'}{3}}+\sqrt[4]{\frac{2l}{3}}+\left(\frac{3}{2^9}e^{8\lambda_1 t}+\frac{1}{3\cdot 2^7}e^{8\lambda_2t}\right)^\frac14\right),\\
	\frac{|\nabla h|}{\sqrt{h}} \le 3\sqrt{\frac{\kappa_2}{\tilde{\kappa}_2}}\left(\sqrt[4]{\frac{l}{3}}+\sqrt[4]{\frac{2l'}{3}}+\left(\frac{1}{2^8}e^{8\lambda_1t}+\frac{1}{3\cdot 2^6}e^{8\lambda_2t}\right)^\frac14\right),
	\end{cases}
	\end{eqnarray}
	where $l=\frac{(\Omega_1^*)^2}{2}+\frac{27}{256}(\Omega_2^*)^4+\frac{3}{2^\frac23}p^\frac83 \kappa_2^{8p-4}\kappa_1^4$ and $l'=\frac{(\bar{\Omega}_1^*)^2}{2}+\frac{27}{256}(\bar{\Omega}_2^*)^4+\frac{3}{2^\frac23}q^\frac83 \kappa_1^{8q-4}\kappa_2^4$.
\end{theorem}
\begin{proof}Let $\mathcal{G}_1$, $\mathcal{G}_2$ achieve maximum at $(x_1,t_1)\in D_{T_1}(2R)$. We assume that at $(x_1,t_1)$, $\mathcal{G}_1\ge 0$, $\mathcal{G}_2\ge 0$, otherwise the proof will be trivial. Hence at that point
\begin{eqnarray}\label{3.13}
	\nabla \mathcal{G}_1 =0,\ \Delta \mathcal{G}_1\le 0,\ \partial_t \mathcal{G}_1\ge 0
\end{eqnarray}
and
\begin{eqnarray}\label{3.14}
	\nabla \mathcal{G}_2 =0,\ \Delta \mathcal{G}_2\le 0,\ \partial_t \mathcal{G}_2\ge 0.
\end{eqnarray}
For ease of calculation we proceed with (\ref{3.13}) and get 
\begin{eqnarray}
	\nabla\mathcal{H}_1 = -\frac{\mathcal{H}_1}{\eta}\nabla\eta.
\end{eqnarray}
By \cite{JSUN}, there is a constant $c_2$ such that
\begin{eqnarray}
	\nonumber -\mathcal{H}_1\eta_t\ge-c_2k_2\mathcal{H}_1.
\end{eqnarray}
Hence
\begin{eqnarray}\label{3.16}
\nonumber	0 &\ge& (\Delta_\phi -\partial_t)\mathcal{G}_1\\
\nonumber	&=& \mathcal{H}_1(\Delta_\phi-\partial_t)\eta+\eta(\Delta_\phi-\partial_t)\mathcal{H}_1\\
	 		&\ge& -\Omega \mathcal{H}_1+\eta(\Delta_\phi-\partial_t)\mathcal{H}_1
\end{eqnarray}
Using (\ref{lemma eqn u}) in (\ref{3.16}) we have 
\begin{eqnarray}\label{3.17}
\nonumber	0 &\ge& -\Omega \mathcal{H}_1 + 4\eta u^{-3}\mathcal{H}_1^2-4\eta u^{-1}\langle\nabla u,\nabla \mathcal{H}_1\rangle -2\eta pe^{\lambda_1t}v^{3p-\frac32}u^{-\frac32}\sqrt{\mathcal{H}_1\mathcal{H}_2}\\
&&+\eta e^{\lambda_1 t}v^{3p}u^{-2}\mathcal{H}_1-2\eta \left((n-1)k_1+k_2\right)\mathcal{H}_1.
\end{eqnarray}
Using the relation $\eta\langle \nabla u,\nabla\mathcal{H}_1\rangle =-\mathcal{H}_1\langle\nabla u,\nabla \eta\rangle \le \frac{\sqrt c_1}{R}\eta^\frac12 \mathcal{H}_1|\nabla u|$ in (\ref{3.17}) and multiplying the resultant equation with $\eta u^3$, we obtain
\begin{eqnarray}\label{3.18}
	\nonumber	0 &\ge& -\Omega u^3\mathcal{G}_1 + 4\mathcal{G}_1^2-4 \frac{\sqrt{c_1}}{R}u^\frac32 \mathcal{G}_1^\frac32 -2pe^{\lambda_1 t}v^{3p-\frac32}u^\frac32\eta \sqrt{\mathcal{G}_1\mathcal{G}_2}+e^{\lambda_1 t}v^{3p}u\eta \mathcal{G}_1\\
	&&-2 ((n-1)k_1+k_2)\eta u^3\mathcal{G}_1,
\end{eqnarray}
or equivalently,
\begin{eqnarray}
	\label{3.19}0 &\ge& 4\mathcal{G}_1^2-\Omega_1\mathcal{G}_1-\Omega_2\mathcal{G}_1^\frac32 -\Omega_3\sqrt{\mathcal{G}_1}.
\end{eqnarray}
Applying Young's inequality to the terms in the right hand side of (\ref{3.19}) we have
\begin{eqnarray}
	\label{3.20}\Omega_1 \mathcal{G}_1 &\le& \frac{(\Omega_1^*)^2}{2}+\frac{\mathcal{G}_1^2}{2},\\
	\label{3.21}\Omega_2\mathcal{G}_1^\frac32 &\le& \frac{27}{256}(\Omega_2^*)^4+\mathcal{G}_1^2,\\
	\label{3.22}\Omega_3\sqrt{\mathcal{G}_1} &\le& \frac{3}{2^\frac23}p^\frac83 \kappa_2^{8p-4}\kappa_1^4+\frac{1}{2^7}e^{8\lambda_1t}+\mathcal{G}_2^2+\frac{1}{2}\mathcal{G}_1^2.
\end{eqnarray}
Using (\ref{3.20}), (\ref{3.21}) and (\ref{3.22}) in (\ref{3.19}) we get
\begin{eqnarray}
	\label{3.23}2\mathcal{G}_1^2 &\le& \frac{(\Omega_1^*)^2}{2}+\frac{27}{256}(\Omega_2^*)^4+\frac{3}{2^\frac23}p^\frac83 \kappa_2^{8p-4}\kappa_1^4+\frac{1}{2^7}e^{8\lambda_1t}+\mathcal{G}_2^2.
\end{eqnarray}
Similarly, using (\ref{3.14}) it can be showed that
\begin{eqnarray}
	\label{3.24}2\mathcal{G}_2^2 &\le& \frac{(\bar{\Omega}_1^*)^2}{2}+\frac{27}{256}(\bar{\Omega}_2^*)^4+\frac{3}{2^\frac23}q^\frac83 \kappa_1^{8q-4}\kappa_2^4+\frac{1}{2^7}e^{8\lambda_2t}+\mathcal{G}_1^2.
\end{eqnarray}
For convenience, we denote $l=\frac{(\Omega_1^*)^2}{2}+\frac{27}{256}(\Omega_2^*)^4+\frac{3}{2^\frac23}p^\frac83 \kappa_2^{8p-4}\kappa_1^4$ and $l'=\frac{(\bar{\Omega}_1^*)^2}{2}+\frac{27}{256}(\bar{\Omega}_2^*)^4+\frac{3}{2^\frac23}q^\frac83 \kappa_1^{8q-4}\kappa_2^4$, thus (\ref{3.23}) and (\ref{3.24}) becomes
\begin{eqnarray}
\label{new 3.28}2\mathcal{G}_1^2 &\le& l+\frac{1}{2^7}e^{8\lambda_1t}+\mathcal{G}_2^2,\\
\label{new 3.29}2\mathcal{G}_2^2 &\le& l'+\frac{1}{2^7}e^{8\lambda_2t}+\mathcal{G}_1^2.
\end{eqnarray}
Applying (\ref{new 3.29}) in (\ref{new 3.28}) and using the result in (\ref{new 3.29}) we deduce 
\begin{eqnarray}\label{new 3.30}
\begin{cases}
\mathcal{G}_1^2 \le \frac13 l'+\frac23 l+\Psi_1(t),\\
\mathcal{G}_2^2 \le \frac13 l +\frac23 l'+\Psi_2(t),
\end{cases}
\end{eqnarray}
where $\Psi_1(t)=\frac{3}{2^9}e^{8\lambda_1 t}+\frac{1}{3\cdot 2^7}e^{8\lambda_2t}$ and $\Psi_2(t)=\frac{1}{2^8}e^{8\lambda_1t}+\frac{1}{3\cdot 2^6}e^{8\lambda_2t}$.
Following \cite{SAZAMI}, $\tilde{\kappa}_1\le u$ and $\tilde{\kappa}_2\le v$ implies $\tilde{\kappa}_1^2\mathcal{H}_1^2\le u^2\mathcal{G}_1^2$ and $\tilde{\kappa}_2^2\mathcal{H}_2^2\le v^2\mathcal{G}_2^2$ respectively. Hence (\ref{new 3.30}) reduces to 
\begin{eqnarray}\label{new 3.31}
\begin{cases}
\mathcal{H}_1^2 \le \frac{\kappa_1^2}{\tilde{\kappa}_1^2}\left(\frac13 l'+\frac23 l+\Psi_1(t)\right),\\
\mathcal{H}_2^2 \le \frac{\kappa_2^2}{\tilde{\kappa}_2^2}\left(\frac13 l+\frac23 l'+\Psi_2(t)\right).
\end{cases}
\end{eqnarray}
Given that $\mathcal{H}_1=u|\nabla u|^2$, $\mathcal{H}_2=v|\nabla v|^2$ with $u=f^\frac13$ and $v=h^\frac13$, so (\ref{new 3.31}) becomes
\begin{eqnarray}\label{new 3.32}
\begin{cases}
\left(\frac19 \frac{|\nabla f|^2}{f}\right)^2 \le \frac{\kappa_1^2}{\tilde{\kappa}_1^2}\left(\frac13 l'+\frac23 l+\Psi_1(t)\right),\\
\left(\frac19 \frac{|\nabla h|^2}{h}\right)^2 \le \frac{\kappa_2^2}{\tilde{\kappa}_2^2}\left(\frac13 l+\frac23 l'+\Psi_2(t)\right).
\end{cases}
\end{eqnarray}
Applying the elementary inequality $\sqrt{x+y}\le\sqrt{x}+\sqrt{y}$ for positive $x,y$, in (\ref{new 3.32}) we get
\begin{eqnarray}\label{new 3.33}
	\begin{cases}
	\frac{|\nabla f|}{\sqrt{f}} \le 3\sqrt{\frac{\kappa_1}{\tilde{\kappa}_1}}\left(\sqrt[4]{\frac{l'}{3}}+\sqrt[4]{\frac{2l}{3}}+\sqrt[4]{\Psi_1(t)}\right),\\
	\frac{|\nabla h|}{\sqrt{h}} \le 3\sqrt{\frac{\kappa_2}{\tilde{\kappa}_2}}\left(\sqrt[4]{\frac{l}{3}}+\sqrt[4]{\frac{2l'}{3}}+\sqrt[4]{\Psi_2(t)}\right).
	\end{cases}
\end{eqnarray}
Using the definition of $\Psi_1(t),\Psi_2(t)$ on the above relations completes the proof.
\end{proof}
\noindent Now we derive global Hamilton type estimate using Theorem \ref{Theorem 3.1}.
\begin{corollary}[Global Hamilton type estimate]
	If $(f,h)$ is a positive solution to the system (\ref{heat eqn 1.1}) along the flow (\ref{flow 1.2}) satisfying $\tilde{\kappa}_1^3\le f\le \kappa_1^3$ and $\tilde{\kappa}_2^3\le h\le \kappa_2^3$ in $M\times [0,T]$ and $Ric_\phi\ge -(n-1)k_1g$ and $\mathcal{S}\ge -k_2 g$ on $M\times[0,T]$ then
	\begin{eqnarray}\label{new corr 3.34}
		\begin{cases}
		\frac{|\nabla f|}{\sqrt{f}} \le 3\sqrt{\frac{\kappa_1}{\tilde{\kappa}_1}}\left(\sqrt[4]{\frac{L'}{3}}+\sqrt[4]{\frac{2L}{3}}+\left(\frac{3}{2^9}e^{8\lambda_1 t}+\frac{1}{3\cdot 2^7}e^{8\lambda_2t}\right)^\frac14\right),\\
		\frac{|\nabla h|}{\sqrt{h}} \le 3\sqrt{\frac{\kappa_2}{\tilde{\kappa}_2}}\left(\sqrt[4]{\frac{L}{3}}+\sqrt[4]{\frac{2L'}{3}}+\left(\frac{1}{2^8}e^{8\lambda_1t}+\frac{1}{3\cdot 2^6}e^{8\lambda_2t}\right)^\frac14\right),
		\end{cases}
	\end{eqnarray}
	where $t\in [0,T]$,
		$L=\frac{1}{2}(\hat{\Omega}_1^*)^2+\frac{3}{2^\frac23}p^\frac83 \kappa_2^{8p-4}\kappa_1^4$ and $l'=\frac{1}{2}(\hat{\bar{\Omega}}_1^*)^2+\frac{3}{2^\frac23}q^\frac83 \kappa_1^{8q-4}\kappa_2^4$.
\end{corollary}
\begin{proof}
	Taking $R\to +\infty$ in (\ref{Theorem 3.1 result})  we obtain (\ref{new corr 3.34}).
\end{proof}
\section{Souplet-Zhang type gradient estimation}
In this section we derive souplet-Zhang type estimation for positive solution of (\ref{heat eqn 1.1}) along (\ref{flow 1.2}). Thus throughout this section we consider $(f,h)=(e^u,e^v)$ as a positive solution to the system of equations (\ref{heat eqn 1.1}). Putting $f=e^u,\ h=e^v$ in (\ref{heat eqn 1.1}) we have
\begin{eqnarray}\label{2.1 Reduced Heat eq-1}
	\begin{cases}
		(\Delta_\phi -\partial_t)u = -|\nabla u|^2-e^{\lambda_1t+vp-u}\\
		(\Delta_\phi -\partial_t)v = -|\nabla v|^2-e^{\lambda_2t+uq-v}.
	\end{cases}
\end{eqnarray}
Let $\bar{u}=-e^{\lambda_1t+vp-u}$ and $\bar{v}=-e^{\lambda_2t+uq-v}$ then system (\ref{2.1 Reduced Heat eq-1}) reduces to 
\begin{eqnarray}\label{2.1 Reduced Heat eq-2}
	\begin{cases}
		(\Delta_\phi -\partial_t)u = -|\nabla u|^2+\bar{u}\\
		(\Delta_\phi -\partial_t)v = -|\nabla v|^2+\bar{v}.
	\end{cases}
\end{eqnarray}
Let $\kappa_1,\tilde{\kappa}_1,\kappa_2,\tilde{\kappa}_2$ be positive constants. Define $\rho_1=1+\kappa_1$ and $\rho_2=1+\kappa_2$. In this section we assume that a positive solution $(f,h)$ of the system (\ref{heat eqn 1.1}) satisfies $\tilde{\kappa}_1\le f\le \kappa_1$ and $\tilde{\kappa}_2\le h\le \kappa_2$. Hence a solution $(u,v)$ of (\ref{2.1 Reduced Heat eq-2}) satisfies $\ln\tilde{\kappa}_1\le u\le \ln\kappa_1$ and $\ln\tilde{\kappa}_2\le v\le \ln\kappa_2$. Let $x_0\in M$ and $R>0$ be any real number.
\begin{lemma}\label{Lemma 1 main}
	Let $(u,v)$ be a solution to the equation (\ref{2.1 Reduced Heat eq-2}). If there exist positive constants $k_1$, $k_2$ such that
	\begin{eqnarray*}
		Ric_\phi\ge -(n-1)k_1 g\text{ and } \mathcal{S}\ge -k_2 g
	\end{eqnarray*} 
	on $D_T(2R)$, then the function $\mathcal{W}_1 := \frac{|\nabla u|^2}{(\rho_1-u)^2}$ satisfies
	\begin{eqnarray}\label{W_1 eqn}
		\nonumber	(\Delta_\phi-\partial_t)\mathcal{W}_1 &\ge& 2\left(\frac{u-\ln\kappa_1}{\rho_1-u}\right)\langle\nabla \mathcal{W}_1,\nabla u\rangle+\frac{2|\nabla u|^4}{(\rho_1-u)^3}-\frac{2|\nabla u|^2}{(\rho_1-u)^3}\bar{u}\\
		\nonumber && -\frac{2\bar{u}}{(\rho_1-u)^2}\left(p\langle\nabla v,\nabla u\rangle-|\nabla u|^2\right)-\frac{2}{(\rho_1-u)^2}((n-1)k_1+k_2)|\nabla u|^2,\\
	\end{eqnarray}
	and the function $\mathcal{W}_2 := \frac{|\nabla v|^2}{(\rho_2-v)^2}$ satisfies
	\begin{eqnarray}\label{W_2 eqn}
		\nonumber	(\Delta_\phi-\partial_t)\mathcal{W}_2 &\ge& 2\left(\frac{v-\ln\kappa_2}{\rho_2-v}\right)\langle\nabla \mathcal{W}_2,\nabla v\rangle+\frac{2|\nabla v|^4}{(\rho_2-v)^3}-\frac{2|\nabla v|^2}{(\rho_2-v)^3}\bar{v}\\
		\nonumber && -\frac{2\bar{v}}{(\rho_2-v)^2}\left(q\langle\nabla u,\nabla v\rangle-|\nabla v|^2\right)-\frac{2}{(\rho_2-v)^2}(k_1+k_2)|\nabla v|^2.\\
	\end{eqnarray}
\end{lemma}
\begin{proof}
	Using weighted Bochner formula (Lemma \ref{lemma Bochner}) we have
	\begin{eqnarray}\label{SZ 3.5}
		\nonumber \Delta_\phi\mathcal{W}_1 &=&\frac{6|\nabla u|^4}{(\rho_1-u)^4}+\frac{2|\nabla u|^2}{(\rho_1-u)^3}\Delta_\phi u+\frac{4}{(\rho_1-u)^3}\langle\nabla|\nabla u|^2,\nabla u\rangle\\
		&&+\frac{2}{(\rho_1-u)^2}\left(|\text{Hess }u|^2+\langle\nabla\Delta_\phi u,\nabla u\rangle+Ric_\phi(\nabla u,\nabla u)\right).
	\end{eqnarray}
	By Lemma \ref{lemma evol eqn} we find
	\begin{eqnarray}\label{SZ 3.6}
		\partial_t \mathcal{W}_1 &=& -\frac{2}{(\rho_1-u)^2}\mathcal{S}(\nabla u,\nabla u)+\frac{2}{(\rho_1-u)^2}\langle\nabla u_t,\nabla u\rangle+\frac{2|\nabla u|^2}{(\rho_1-u)^3}u_t.
	\end{eqnarray}
	Subtracting (\ref{SZ 3.6}) from (\ref{SZ 3.5}) then applying (\ref{2.1 Reduced Heat eq-2}) we get
	\begin{eqnarray}\label{SZ 3.7}
		\nonumber (\Delta_\phi-\partial_t)\mathcal{W}_1 &=& \frac{6|\nabla u|^4}{(\rho_1-u)^4}-\frac{2|\nabla u|^4}{(\rho_1-u)^3}+\frac{2|\nabla u|^2}{(\rho_1-u)^3}\bar{u}+\frac{8\text{Hess }u(\nabla u,\nabla u)}{(\rho_1-u)^3}\\
		\nonumber &&+\frac{2|\text{Hess }u|^2}{(\rho_1-u)^2}+\frac{2}{(\rho_1-u)^2}(Ric_\phi+\mathcal{S})(\nabla u,\nabla u)\\
		&&-\frac{4\text{Hess }u(\nabla u,\nabla u)}{(\rho_1-u)^2}+\frac{2}{(\rho_1-u)^2}\langle\nabla \bar{u},\nabla u\rangle.
	\end{eqnarray}
	Note that
	\begin{eqnarray}
		\nonumber	\frac{|\text{Hess }u|^2}{(\rho_1-u)^2}+\frac{2\text{Hess }u(\nabla u,\nabla u)}{(\rho_1-u)^3}+\frac{|\nabla u|^4}{(\rho_1-u)^4}=\frac{1}{(\rho_1-u)^2}\left|\text{Hess }u+\frac{\nabla u\otimes \nabla u}{(\rho_1-u)^2}\right|^2\ge 0
	\end{eqnarray}
	and
	\begin{eqnarray}
		\nonumber \frac{2\text{Hess }u(\nabla u,\nabla u)}{(\rho_1-u)^2}+\frac{2|\nabla u|^4}{(\rho_1-u)^3}&=&\langle\nabla\mathcal{W}_1,\nabla u\rangle.
	\end{eqnarray}
	Hence (\ref{SZ 3.7}) reduces to 
	\begin{eqnarray}
		\nonumber (\Delta_\phi-\partial_t)\mathcal{W}_1 &\ge& \frac{2\langle\nabla\mathcal{W}_1,\nabla u\rangle}{\rho_1-u}-2\langle\nabla\mathcal{W}_1,\nabla u\rangle+\frac{2|\nabla u|^4}{(\rho_1-u)^3}-\frac{2|\nabla u|^2\bar{u}}{(\rho_1-u)^3}\\
		&& -\frac{2}{(\rho_1-u)^2}\langle\nabla \bar{u},\nabla u\rangle+\frac{2}{(\rho_1-u)^2}(Ric_\phi+\mathcal{S})(\nabla u,\nabla u)
	\end{eqnarray}
	or,
	\begin{eqnarray}\label{SZ 3.9}
		\nonumber (\Delta_\phi-\partial_t)\mathcal{W}_1 &\ge& 2\langle\nabla\mathcal{W}_1,\nabla u\rangle\left(\frac{u-\ln\kappa_1}{\rho_1-u}\right)+\frac{2|\nabla u|^4}{(\rho_1-u)^3}-\frac{2|\nabla u|^2\bar{u}}{(\rho_1-u)^3}\\
		\nonumber && -\frac{2\bar{u}}{(\rho_1-u)^2}(p\langle\nabla v,\nabla u\rangle-|\nabla u|^2)+\frac{2}{(\rho_1-u)^2}(Ric_\phi+\mathcal{S})(\nabla u,\nabla u).\\
	\end{eqnarray}
	We apply the bounds of $Ric_\phi$ and $\mathcal{S}$ in the above equation to find (\ref{W_1 eqn}). \\
	\\
	\noindent Using the symmetry in (\ref{2.1 Reduced Heat eq-2}) and applying similar method in (\ref{W_1 eqn}) we obtain (\ref{W_2 eqn}). This completes the proof.
\end{proof}
\begin{theorem}\label{Theorem 4.1}
	Let $(f,h)=(e^u,e^v)$ be a solution to the equation (\ref{heat eqn 1.1}). If there exist positive constants $k_1$, $k_2$ such that
	\begin{eqnarray*}
		Ric_\phi\ge -(n-1)k_1 g\text{ and } \mathcal{S}\ge -k_2 g
	\end{eqnarray*} 
	on $D_T(2R)$, then
	\begin{eqnarray}\label{new SZ Theorem 4.1}
	\begin{cases}
	\frac{|\nabla f|}{f} \le \left(1+\ln(\frac{\kappa_1}{\tilde{\kappa}_1})\right)\left(\frac{20^\frac14}{\sqrt 3}b+\frac{2}{\sqrt3}b'\right),\\
	\frac{|\nabla h|}{h} \le \left(1+\ln(\frac{\kappa_2}{\tilde{\kappa}_2})\right)\left(\frac{2}{\sqrt3}b+\frac{20^\frac14}{\sqrt2}b'\right),
	\end{cases}
	\end{eqnarray}
	where $b=\sqrt{\Lambda_1}+\frac{27^\frac14}{\sqrt2}\Lambda_2+\sqrt{p\bar{u}^*}$ and $b'=\sqrt{\bar{\Lambda}_1}+\frac{27^\frac14}{\sqrt2}\Lambda_2+ \sqrt{q\bar{v}^*}$ with
	\begin{enumerate}
	\item[] $\Lambda_1=\frac{c_0}{R}(n-1)(\sqrt{k_1}+\frac2R+\frac{3c_1}{R^2}+c_2k_2)+4\bar{u}^*+2((n-1)k_1+k_2)$,
	\item[] $\Lambda_2=2\ln\left(\frac{\kappa_1}{\tilde{\kappa}_1}\right)\frac{\sqrt{c_1}}{R}$,
	\item[] $\bar{\Lambda}_1=\frac{c_0}{R}(n-1)(\sqrt{k_1}+\frac2R+\frac{3c_1}{R^2}+c_2k_2)+4\bar{v}^*+2((n-1)k_1+k_2)$,
	\item[] $\bar{u}^*=e^{\lambda_1 t+p\tilde{\kappa}_2-\kappa_1}$, 
	\item[] $\bar{v}^*=e^{\lambda_2 t+q\tilde{\kappa}_1-\kappa_2}$ and $t\in[0,T]$.
	\end{enumerate}
\end{theorem}
\begin{proof}
	Consider $\psi$ and $\eta$ as in (\ref{eq psi defn}), (\ref{eq eta}). Let $G_1=\eta\mathcal{W}_1$ and $G_2=\eta \mathcal{W}_2$. Fix $T_2\in(0,T]$ and assume $G_i$ achieves maximum at $(x_0,t_0)\in D_{T_2}(2R)$ with $G_i(x_0,t_0)\ge0$ (if $G_i(x_0,t_0)\le 0$ then the proof is trivial) for $i=1,2$.\\
	\noindent Hence at $(x_0,t_0)$ we have
	\begin{eqnarray}
		\label{SZ 3.10}	\nabla G_1 =0,\ \Delta G_1 \le 0,\ \partial_tG_1\ge 0,\\
		\label{SZ 3.11}	\nabla G_2 =0,\ \Delta G_2 \le 0,\ \partial_tG_2\ge 0.
	\end{eqnarray}
	Therefore,
	\begin{eqnarray}
		\label{SZ 3.12}\nabla\mathcal{W}_1=-\frac{\mathcal{W}_1}{\eta}\nabla\eta,\\
		\label{SZ 3.13}\nabla\mathcal{W}_2=-\frac{\mathcal{W}_2}{\eta}\nabla\eta
	\end{eqnarray}
	and
	\begin{eqnarray}
		\label{SZ 3.14}	0\ge (\Delta_\phi -\partial_t)G_1 =\mathcal{W}_1(\Delta_\phi -\partial_t)\eta+\eta(\Delta_\phi -\partial_t)\mathcal{W}_1+2\langle\nabla\mathcal{W}_1,\nabla \eta\rangle,\\
		\label{SZ 3.15}	0\ge (\Delta_\phi -\partial_t)G_2 =\mathcal{W}_2(\Delta_\phi -\partial_t)\eta+\eta(\Delta_\phi -\partial_t)\mathcal{W}_2+2\langle\nabla\mathcal{W}_2,\nabla \eta\rangle.
	\end{eqnarray}
	By \cite{JSUN}, there is a constant $c_2$ such that
	\begin{eqnarray}
		\label{SZ 3.16} -\mathcal{W}_1\eta_t\ge -c_2k_2\mathcal{W}_1,\\
		\label{SZ 3.17} -\mathcal{W}_2\eta_t\ge -c_2k_2\mathcal{W}_2.
	\end{eqnarray}
	For ease of calculation we now proceed with (\ref{SZ 3.10}), (\ref{SZ 3.12}), (\ref{SZ 3.14}) and (\ref{SZ 3.16}). Using generalized Laplacian comparison theorem, (\ref{SZ 3.12}) and (\ref{SZ 3.16}) in (\ref{SZ 3.14}) we get
	\begin{eqnarray}
		\label{SZ 3.18}0 &\ge& -\Omega \mathcal{W}_1+\eta(\Delta_\phi-\partial_t)\mathcal{W}_1,
	\end{eqnarray}
	where $\Omega=\frac{c_0}{R}(n-1)(\sqrt{k_1}+\frac2R+\frac{3c_1}{R^2}+c_2k_2)$.\\
	Using (\ref{W_1 eqn}) from Lemma \ref{Lemma 1 main} in (\ref{SZ 3.18}) we obtain
	\begin{eqnarray}
		\nonumber\label{SZ 3.19}0 &\ge& -\Omega \mathcal{W}_1+2\eta\left(\frac{u-\ln\kappa_1}{\rho_1-u}\right)\langle\nabla \mathcal{W}_1,\nabla u\rangle+\frac{2\eta|\nabla u|^4}{(\rho_1-u)^3}-\frac{2\eta|\nabla u|^2}{(\rho_1-u)^3}\bar{u}\\
		\nonumber&& -\frac{2\eta\bar{u}}{(\rho_1-u)^2}\left(p\langle\nabla v,\nabla u\rangle-|\nabla u|^2\right)-\frac{2\eta}{(\rho_1-u)^2}((n-1)k_1+k_2)|\nabla u|^2.\\
	\end{eqnarray}
	Multiplying (\ref{SZ 3.19}) with $\eta$ we get
	\begin{eqnarray}
		\nonumber\label{SZ 3.20}0 &\ge& -\Omega G_1-2\ln\left(\frac{\kappa_1}{\tilde{\kappa}_1}\right)\frac{\sqrt{c_1}}{R}G_1^\frac32+2(\rho_1-u)G_1^2-\frac{2G_1}{\rho_1-u}\bar{u}\\
		\nonumber&& -2\bar{u}\eta^2p\frac{|\nabla v||\nabla u|}{(\rho_1-u)^2}+2\bar{u}\eta^2\frac{|\nabla u|^2}{(\rho_1-u)^2}-2G_1((n-1)k_1+k_2).\\
	\end{eqnarray}
	We see that $\rho_1-u\ge 1$, $\rho_2-v\ge 1$, $\bar{u}\le -e^{\lambda_1 t+p\tilde{\kappa}_2-\kappa_1}$ and $\eta=1$ on $D_{T_2}(2R)$. Thus (\ref{SZ 3.20}) becomes
	\begin{eqnarray}
		\label{SZ 3.21}0 &\ge& 2G_1^2-\Lambda_1 G_1-\Lambda_2 G_1^\frac32-\Lambda_3\sqrt{G_1}.
	\end{eqnarray}
	where 
	$\Lambda_1=\Omega+4e^{\lambda_1 t+p\tilde{\kappa}_2-\kappa_1}+2((n-1)k_1+k_2)$, $\Lambda_2=2\ln\left(\frac{\kappa_1}{\tilde{\kappa}_1}\right)\frac{\sqrt{c_1}}{R}$, $\Lambda_3=2e^{\lambda_1 t+p\tilde{\kappa}_2-\kappa_1}\sqrt{G_2}$.
By using Young's inequality in each terms of (\ref{SZ 3.21}) we obtain the following results.
\begin{eqnarray}
	\label{Young-1} \Lambda_1 G_1 &\le& \Lambda_1^2 +\frac{G_1^2}{4},\\
	\label{Young-2} \Lambda_2G_1^\frac32 &\le& \frac{27}{4}\Lambda_2^4+\frac{G_1^2}{4},\\ 
	\label{Young-3} \Lambda_3G_1^\frac12 &\le& \frac34 \Lambda_3^\frac43 + \frac{G_1^2}{4}.
\end{eqnarray}
Using (\ref{Young-1}), (\ref{Young-2}) and (\ref{Young-3}) in (\ref{SZ 3.21}) we get
\begin{eqnarray}
	\frac{5}{4}G_1^2 &\le& \Lambda_1^2+\frac{27}{4}\Lambda_2^4+\frac34 \Lambda_3^\frac43.
\end{eqnarray}
Applying Young's inequality on $\frac34 \Lambda_3^\frac43$ we have
\begin{eqnarray}
	\label{SZ 3.26} \frac{5}{4}G_1^2 &\le& \Lambda_1^2+\frac{27}{4}\Lambda_2^4+ p^2 (\bar{u}^*)^2+G_2^2,
\end{eqnarray}
where $\bar{u}^*=e^{\lambda_1 t+p\tilde{\kappa}_2-\kappa_1}$.
\\
Similarly from (\ref{SZ 3.11}), (\ref{SZ 3.13}), (\ref{SZ 3.15}) and (\ref{SZ 3.17}), we obtain
\begin{eqnarray}
	\label{SZ 3.27} \frac{5}{4}G_2^2 &\le& \bar{\Lambda}_1^2+\frac{27}{4}\Lambda_2^4+ q^2 (\bar{v}^*)^2+G_1^2,
\end{eqnarray}
where $\bar{\Lambda}_1=\Omega+4e^{\lambda_2t+q\tilde{\kappa}_1-\kappa_2}+2((n-1)k_1+k_2)$ and $\bar{v}^*=e^{\lambda_2t+q\tilde{\kappa}_1-\kappa_2}$.\\
Using (\ref{SZ 3.27}) in (\ref{SZ 3.26}) we conclude 
\begin{eqnarray}
	\label{SZ 3.28} G_1^2 &\le& \frac{20}{9}\left(\Lambda_1^2+\frac{27}{4}\Lambda_2^4+ p^2 (\bar{u}^*)^2\right)+\frac{16}{9}\left(\bar{\Lambda}_1^2+\frac{27}{4}\Lambda_2^4+ q^2 (\bar{v}^*)^2\right)
\end{eqnarray}
and using (\ref{SZ 3.28}) in (\ref{SZ 3.27}) we deduce
\begin{eqnarray}
	\label{SZ 3.29} G_2^2 &\le& \frac{16}{9}\left(\Lambda_1^2+\frac{27}{4}\Lambda_2^4+ p^2 (\bar{u}^*)^2\right)+\frac{20}{9}\left(\bar{\Lambda}_1^2+\frac{27}{4}\Lambda_2^4+ q^2 (\bar{v}^*)^2\right).
\end{eqnarray}
Since $\eta=1$ on $D_{T}(2R)$, so putting $u=\ln f$ in (\ref{SZ 3.28}), $v=\ln h$ in (\ref{SZ 3.29}) we get
\begin{eqnarray}
	\nonumber\label{SZ 3.30}\left(\frac{|\nabla f|^2}{f^2(\rho_1-\ln f)}\right)^2 &\le& \frac{20}{9}\left(\Lambda_1^2+\frac{27}{4}\Lambda_2^4+ p^2 (\bar{u}^*)^2\right)+\frac{16}{9}\left(\bar{\Lambda}_1^2+\frac{27}{4}\Lambda_2^4+ q^2 (\bar{v}^*)^2\right),\\
	\\
	\nonumber\label{SZ 3.31}\left(\frac{|\nabla h|^2}{h^2(\rho_2-\ln h)}\right)^2 &\le& \frac{16}{9}\left(\Lambda_1^2+\frac{27}{4}\Lambda_2^4+ p^2 (\bar{u}^*)^2\right)+\frac{20}{9}\left(\bar{\Lambda}_1^2+\frac{27}{4}\Lambda_2^4+ q^2 (\bar{v}^*)^2\right).\\
\end{eqnarray}
Applying the inequality $\sqrt{x+y}\le\sqrt{x}+\sqrt{y}$, for positive $x,y$ twice in each of the above equations and using the fact that $\rho_1-\ln f\le 1+\ln\left(\frac{\kappa_1}{\tilde{\kappa}_1}\right)$, $\rho_2-\ln h\le 1+\ln\left(\frac{\kappa_2}{\tilde{\kappa}_2}\right)$ we obtain
\begin{eqnarray}
		\label{SZ 3.32}\frac{|\nabla f|}{f} &\le& \left(1+\ln(\frac{\kappa_1}{\tilde{\kappa}_1})\right)\left(\frac{20^\frac14}{\sqrt 3}b+\frac{2}{\sqrt3}b'\right),\\
		\label{SZ 3.33}\frac{|\nabla h|}{h} &\le& \left(1+\ln(\frac{\kappa_2}{\tilde{\kappa}_2})\right)\left(\frac{2}{\sqrt3}b+\frac{20^\frac14}{\sqrt2}b'\right),
\end{eqnarray}
where $b=\sqrt{\Lambda_1}+\frac{27^\frac14}{\sqrt2}\Lambda_2+\sqrt{p\bar{u}^*}$ and $b'=\sqrt{\bar{\Lambda}_1}+\frac{27^\frac14}{\sqrt2}\Lambda_2+ \sqrt{q\bar{v}^*}$. This completes the proof.
\end{proof}
\noindent As a result we have the global gradient estimate for (\ref{heat eqn 1.1}) along (\ref{flow 1.2}) on $M$.
\begin{corollary}[Global Souplet-Zhang type estimate]
If $(f,h)$ is a positive solution to the system (\ref{heat eqn 1.1}) along the flow (\ref{flow 1.2}) satisfying $\tilde{\kappa}_1^3\le f\le \kappa_1^3$ and $\tilde{\kappa}_2^3\le h\le \kappa_2^3$ in $M\times [0,T]$ and $Ric_\phi\ge -(n-1)k_1g$ and $\mathcal{S}\ge -k_2 g$ on $M\times[0,T]$ then
\begin{eqnarray}\label{new SZ corollary 4.1}
\begin{cases}
\frac{|\nabla f|}{f} \le \left(1+\ln(\frac{\kappa_1}{\tilde{\kappa}_1})\right)\left(\frac{20^\frac14}{\sqrt 3}B+\frac{2}{\sqrt3}B'\right),\\
\frac{|\nabla h|}{h} \le \left(1+\ln(\frac{\kappa_2}{\tilde{\kappa}_2})\right)\left(\frac{2}{\sqrt3}B+\frac{20^\frac14}{\sqrt2}B'\right),
\end{cases}
\end{eqnarray}
where $B=\sqrt{\hat{\Lambda}_1}+\sqrt{p\bar{u}^*}$ and $B'=\sqrt{\hat{\bar{\Lambda}}_1}+ \sqrt{q\bar{v}^*}$ with
\begin{enumerate}
	\item[] $\hat{\Lambda}_1=c_2k_2+4\bar{u}^*+2((n-1)k_1+k_2)$,
	\item[] $\hat{\bar{\Lambda}}_1=c_2k_2+4\bar{v}^*+2((n-1)k_1+k_2)$,
	\item[] $\bar{u}^*=e^{\lambda_1 t+p\tilde{\kappa}_2-\kappa_1}$, 
	\item[] $\bar{v}^*=e^{\lambda_2 t+q\tilde{\kappa}_1-\kappa_2}$ and $t\in[0,T]$.
\end{enumerate}
\end{corollary}
\begin{proof}
From Theorem \ref{Theorem 4.1} we have the local gradient estimate for positive solutions of (\ref{heat eqn 1.1}) along (\ref{flow 1.2}). Taking $R\to +\infty$ in (\ref{new SZ Theorem 4.1}) we get (\ref{new SZ corollary 4.1}). 
\end{proof}
\section{Concluding remark}
To find an exact solution of a partial differential equation (PDE) in higher dimensions turns out to be impossible in most of the cases. Gradient estimation allows us to uncover possible nature of the solutions for certain PDE's, without calculating its exact solution. In this article we have studied the system (\ref{heat eqn 1.1}) along (\ref{flow 1.2}) on a weighted Riemannian manifold and derived certain constants to bound the quantities $\frac{|\nabla f|}{\sqrt{f}}$, $\frac{|\nabla h|}{\sqrt{h}}$, $\frac{|\nabla f|}{f}$ and $\frac{|\nabla h|}{ h}$ in a definite way, where $(f,h)$ is a positive solution of (\ref{heat eqn 1.1}). This allows us to get an idea about the possible upper and lower bound for the gradient of those solutions.\\
\\
\vspace{0.1in}\noindent{\bf Acknowledgement:} This research work is partially supported by DST FIST programme (No. SR/FST/MSII/2017/10(C)).

\vspace{0.1in}
\noindent Shyamal Kumar Hui\\
\noindent Department of Mathematics, The University of Burdwan, Golapbag, Burdwan 713104, West Bengal, India. \\
\noindent E-mail: skhui@math.buruniv.ac.in\\

\noindent Shahroud Azami\\
\noindent Department of Pure Mathematics, Faculty of Sciences, Imam Khomeini International University, Qazvin, Iran.\\
\noindent E-mail: azami@sci.ikiu.ac.ir\\

\noindent Sujit Bhattacharyya\\
Department of Mathematics, The University of Burdwan, Golapbag, Burdwan 713104, West Bengal, India. \\
\noindent E-mail: sujitbhattacharyya.1996@gmail.com
\end{document}